\documentclass[11pt]{amsart}
\usepackage{amssymb,amsmath,amsthm,latexsym}
\usepackage{amssymb,graphicx}
\newtheorem{theorem}{Theorem}

\newtheorem{lemma}[theorem]{Lemma}

\newtheorem{question}[theorem]{Question}
\newtheorem{proposition}[theorem]{Proposition}

\title{Growth of homology torsion of metabelian groups}
\author{Nikolay Nikolov}
\begin{document}
\begin{abstract} We study the growth of torsion of the abelianization  of finite index subgroups in finitely generated metabelian groups. This complements the results in \cite{KKN} which covered the finitely presented amenable groups.
\end{abstract}

\maketitle
\section{Introduction}

For subgroups $A, B$ of a group $G$ we denote by $[A,B]$ the subgroup of $G$ generated by all commutators $[a,b]$ with $a \in A, b \in B$. Set $G'=[G,G]$ and $G^{ab}=G/G' \simeq H_1(G,\mathbb Z)$. 
Let $d(G)$ denote the minimal size of a generating set of a group $G$. Let $t(G)$ be the maximal size of a finite subgroup of $G$ (where we set $t(G)= \infty$ if there is no such maximum).
Note that when $G$ is a finitely generated abelian group then $t(G)$ is the size of the torsion subgroup of $G$ and is always finite. \medskip

Given a finitely generated group $G$ and a subgroup $H$ of finite index in $G$ we are interested in the growth of $t(H^{ab})$ in terms of $[G:H]$. This topic is at the crossroads of group theory, operator algebras, geometry and number theory and has intriguing open questions, see \cite{AGN}, \cite{berg}, \cite{WL}. It is easy to see (cf. \cite{AGN}, Lemma 27, restated as Lemma \ref{fp} below) that if $G$ is a finitely presented group then $t(H^{ab})$ is bounded above by an exponential function in $[G:H]$.
In case $G$ is a finitely presented amenable group and $(H_i)$ is a Farber chain in $G$ (for example if $(H_i)$ are normal subgroups of $G$ with trivial intersection) it is proved in  \cite{KKN} that $t(H_i^{ab})$ grows subexponentially in $[G:H_i]$. By way of contrast \cite{AGN} also proves that when $G$ is allowed to vary over all finitely generated solvable groups of derived length 3, then there is no single function $f$ in terms of $[G:H_i]$ which bounds $t(H_i^{ab})$ as $[G:H_i] \rightarrow \infty$. One is thus led to the question what happens for finitely generated metabelian groups which are not finitely presented. It is easy to see that some function bounding $t(H^{ab})$ exists in this class: There are only countably many finitely generated metabelian groups and a diagonal argument produces a function $f : \mathbb N \rightarrow \mathbb N$ with the following property: Given a finitely generated metabelian group $G$ there is $a=a(G) \in \mathbb N$ such that if $H \leq G$ with $a<[G:H] < \infty$, then $t(H^{ab}) < f([G:H])$. It is natural to expect that this non-constructive bound could be improved when we require that the coset space $G/H$ approximates $G$ sufficiently well, for example if $H$ is member of a chain of normal subgroups $(H_i)$ with trivial intersection. Part 1 of Theorem \ref{meta} shows  that the function $f$ above can be taken to be any superexponential functon, e.g. $f(n)=n^n$. 

\begin{theorem} \label{meta} Let $G$ be a finitely generated metabelian group and let $A=G'$.

1. There is a constant $D$ depending on $G$ such that $t(H^{ab}) < D^{[G:H]}$ for every subgroup $H$ of finite index in $G$. 

2.  Let $(H_i)$ be a sequence of finite index subgroups in $G$ such that $[A: (A \cap H_i)] \rightarrow \infty$. Then
\[ \lim_{i \rightarrow \infty} \frac{\log t(H_i^{ab})}{[G:H_i]}=0.\]
\end{theorem}

It is easy to see that the exponential bound in part 1 of Theorem \ref{meta} is sharp.
Let us take $G= C_2 \wr \mathbb Z$ and for $n \in \mathbb N$ let $H_n=\pi^{-1}(n \mathbb Z)$ where $\pi : G \rightarrow \mathbb Z$ is the homomorphim of $G$ onto the top group $\mathbb Z$. Then $t(H_n^{ab}) =2^n$.

As a by-product of our method we can give a short proof of a special case of a theorem of Luck from \cite{WL}.

\begin{theorem}\label{luck} Let $G$ be a finitely presented group with an infinite abelian normal subgroup $A$. Let $(H_i)$ be a sequence of finite index subgroups of $G$ with $[A:(A \cap H_i)] \rightarrow \infty$. Then 
\[ \lim_{i \rightarrow \infty} \frac{\log t(H_i^{ab})}{[G:H_i]}=0.\]
\end{theorem}

Luck's result proves subexponential growth of the torson of integral homology in all degrees (provided $G$ has type $F$) in the more general situation when $G$ has an infinite normal elementary amenable subgroup but under the stronger assumption that $(H_i)$ is a normal chain in $G$ with trivial intersection. 

\subsection*{A question} Let $G$ be an amenable group of type $F_{n+1}$. Corollary 2 of \cite{KKN} proved subexponential growth of $t(H_n(M_i, \mathbb Z))$ for any Farber chain of finite index subgroups $(M_i)$ in $G$. In view of Theorem \ref{meta} we can ask whether in case $G$ is a metabelian group the conclusion holds under the weaker assumption that $G$ is of type $F_n$ or even $FP_{n}$. 
\begin{question} Let $G$ be a metabelian group of type $F_n$ and let $(M_i)$ be a chain of normal subgroups with trivial intersection in $G$. Is it true that
\[ \lim_{i \rightarrow \infty} \frac{t(H_n(M_i, \mathbb Z))}{[G:M_i]}=0?\]
\end{question}
Note that we need that $G$ is at least of type $FP_n$ in order to guarantee that $H_n(G,\mathbb Z)$ is finitely generated.
\section{Proofs}

We begin with some elementary results.

\begin{proposition} \label{t} Let $N$ be a normal subgroup of a group $G$. Then \[ t(N) \leq t(G) \leq t(N) t(G/N).\] If $G/N$ is torsion-free then $t(N)=t(G)$. If $N$ is finite then $t(G)=|N| t(G/N)$.
\end{proposition}
\begin{proof} This is clear. \end{proof}
\begin{lemma} \label{fin}  Let $M$ be a right $\mathbb Z [G]$-module and let $L \leq M$ be a submodule of finite index. Then $[M(G-1): L(G-1)] \leq [M:L]^d$ where $d=d(G)$
\end{lemma}

\begin{proof} Let $g_1, \ldots, g_d$ be a generating set for $G$ of minimal size. Note that
$M(G-1)= \sum_{i=1}^d M(g_i-1)$. Therefore the map $f: M^d \rightarrow M(G-1)$ defined by $f(m_1, \ldots,m_d)= \sum_{i=1}^d m_i (g_i-1)$ ($m_i \in M$) is surjective. Sumilarly $f(L^d)= L(G-1)$ and so $f$ induces an additive group homomorphism \[ \bar f : \left(\frac{M}{L} \right)^d \longrightarrow \frac{M(G-1)}{L(G-1)} \] which is surjective. The claim of the lemma follows.
\end{proof}

The following Lemma is well known (compare with Lemma  6 of \cite{KKN}). For a vector $v =(x_1, \ldots x_n) \in \mathbb Z^n$ we denote by $|v|$ the $l_1$-norm of $v$, i.e. $|v|= \sum_{i=1}^n |x_i|$.
\begin{lemma} \label{torsion} Let $n \in \mathbb N$ and let $v_1, \ldots, v_k \in \mathbb Z^n$.
Let $G= \mathbb Z^n /( \sum_{j=1}^k \mathbb Z v_j)$. Then $t(A) \leq \prod_{i=1}^n c_i$ where $c_1, \ldots c_n$ are the $n$ largest values from the list $|v_1|, \ldots, |v_k|$.
\end{lemma}

\begin{proof}Let $X$ be the $n \times k$ matrix with rows $v_1, \ldots, v_k$. Then $t(A)$ is the g.c.d of the non-zero minors of maximal rank in $X$. Any such minor has rank at most $n$ and so is at most $c_1 \cdots c_n$. \end{proof}

\begin{lemma}[Lemma 27 of \cite{AGN}] \label{fp} Let $G$ be a finitely presented group. There is a constant $C$ depending on $G$ such that $t(H^{ab}) \leq C^{[G:H]}$ for any subgroup $H$ of finite index in $G$.
\end{lemma}

\begin{proposition} \label{MA} Let $G$ be group with a normal abelian subgroup $A$. Let $H$ be a subgroup of finite index in $G$ such that $HA=G$. Then \[t(H^{ab}) \leq [G:H]^{d(G/A)} t(G^{ab}).\]
\end{proposition}
\begin{proof} Let $B= A \cap H$, this is a normal subgroup of $G$ with $[A:B]=[G:H]=n$ say.
Considering $A$ and $B$ as $\mathbb Z [G/A] $-modules (with the action of $G/A$ on $A$ by conjugation) Lemma \ref{fin} gives $[[A,G]:[B,G]] \leq n^d$ where $d=d(G/A)$.

The group $A/[A,G]$ is central in $G/[A,G]$ and expanding $G'=[G,G]=[HA,HA]$ mod $[A,G]$ we obtain $G'=H'[A,G]$. Therefore $[G':H']= [[A,G]: ([A,G] \cap H')] \leq [[A,G]:[B,G]]$ since $B= H \cap A$ and so $H' \geq [B,H]=[B,G]$. Therefore $[G':H'] \leq n^d$ and in particular $[(G' \cap H):H'] \leq n^d$.

Now we can apply Proposition \ref{t} to $H/H'$ with a finite normal subgroup $(G' \cap H)/H'$ and we obtain \[ t(H/H')= |(G' \cap H)/H'| \ t(H/(G' \cap H)) \leq n^d t(H/(G' \cap H)).\] On the other hand $H / (G' \cap H) \simeq HG'/G' \leq G/G'$. Hence $t(H/(G' \cap H))=  t(HG'/G') \leq t(G/G')$ and the result follows.
\end{proof}

\begin{proof}[Proof of Theorem \ref{luck}]. Let $G_i=AH_i$ and $m_i=[G:G_i], a_i=[G_i:H_i]=[A:(A \cap H_i)]$.
Thus $[G:H_i]=m_ia_i$ with $a_i \rightarrow \infty$. Let $d=d(G)$. The Nielsen-Schreier theorem gives $d(G_i) \leq (d-1)m_i+1 \leq dm_i$.  By Lemma \ref{fp} there is a constant $C$ depending on $G$ such that
$t(G_i^{ab}) \leq C^{m_i}$. Now we apply Proposition \ref{MA} to $G_i$ with normal subgroup $A$ and a finite index subgroup $H_i$ (which satisfies $H_iA=G_i$ by the definition of $G_i$). We obtain

\[ t(H^{ab}_i) \leq a_i^{d(G_i/A)} t(G^{ab}_i) \leq a_i^{dm_i}C^{m_i}.\]

Therefore
\begin{equation} \label{x} \frac{\log t(H_i^{ab})}{[G:H_i]} \leq \frac{m_i(d \log a_i + \log C)}{a_im_i}= \frac{d \log a_i+ \log C}{a_i} \rightarrow 0,\end{equation}
since $a_i^{-1}\log a_i \rightarrow 0$ as $a_i \rightarrow \infty$.
\end{proof}

\begin{proof}[Proof of Theorem \ref{meta}]

We need the following.

\begin{proposition} \label{exp} Let $G$ be a finitely generated metabelian group. There is a constant $C$ depending on $G$ with the following property: Let $H$ be a subgroup of finite index in $G$ containing $G'$. Then $t(H^{ab}) \leq C^{[G:H]}$.
\end{proposition}

Let us postpone the proof of Proposition \ref{exp} for the moment and finish the proof of Theorem \ref{meta}.

Define $G_i=AH_i$ and set $m_i=[G:G_i], a_i=[G_i:H_i]=[A: (A \cap H_i)]$ so that $[G:H_i]=a_im_i$. By Proposition \ref{exp} we have $t(G_i^{ab}) \leq C^{m_i}$ for some constant $C$ depending only on $G$. On the other hand we can apply Proposition \ref{MA} to $G_i$ with a normal abelian subgroup $A$ and a finite index subgroup $H_i$ obtaining $t(H_i) \leq a_i^{d(G_i/A)} \  t(G_i^{ab})$. Since $G/A$ is abelian we deduce that $d(G_i/A) \leq 
d(G/A)=d$ say, and so $t(H_i^{ab}) \leq a_i^d C^{m_i}$. Part 1 follows since $a_i^d C^{m_1} \leq 2^{da_i}C^{m_i} \leq (2^dC)^{[G:H_i]}$ and we can take $D=2^dC$.

Part 2 of Theorem \ref{meta} easily follows from a computation similar to (\ref{x}) using that $a_i \rightarrow \infty$.
\end{proof} 

\medskip

It remains to prove Proposition \ref{exp}.
\begin{proof}[Proof of Proposition \ref{exp}]
Let $g_1, \ldots, g_d$ be a generating set of $G$. Let $n=[G:H]$ and note that $H$ is normal in $G$ since $H \geq G'$. Let $L:= G^nG'$, then $L \leq H$ and $[G:L] \leq n^d$. Let $A:=G'$, then $A$ is a $\mathbb Z [G^{ab}]$ module generated by $\{ [g_i,g_j] \ | \ 1 \leq i <j \leq d\}$.
Let $W:=[A,H] \leq A$. The quotient $H/W$ is a finitely generated nilpotent group of class at most 2 and $LW/W$ is a subgroup of finite index in $H/W$. Therefore $(LW/W)'=L'W/W$ has finite index in $(H/W)'=H'W$, i.e. the index $[H': L'W]$ is finite. By Proposition \ref{t} $t(H/L'W)=t(H/H')|H'/LW|$ and in particular $t(H/H') \leq t(H/L'W)$.  In turn $t(H/L'W) \leq t(H/A) t(A/L'W) \leq t(G^{ab})t(A/L'W)$. We will find a bound for $t(A/L'W)$. Note that $A/W$ is in the centre of $L/W$ and $L= \langle g_1^n, \ldots, g_d^n \rangle A$. Hence $L'W= \langle  [g_i^n,g_j^n] \ | \ 1 \leq i <j \leq d \rangle W$.

Let 
\[ X:= \bigoplus_{ 1 \leq i < j \leq d}e_{i,j} \mathbb Z [G^{ab}]\] be the free $\mathbb Z [G^{ab}]$-module with basis $\{e_{i,j} |  \ 1 \leq i < j \leq d\}$ and let $f: X \rightarrow A$ be the surjective $\mathbb Z [G^{ab}]$-module homomorphism such that $f(e_{i,j})=[g_i,g_j]$. Since $X$ is a Noetherian module $\ker f$ is generated (as a $\mathbb Z [G^{ab}]$-module) by finitely many elements, say $\{r_1, \ldots, r_k\} \subset X$.

Let $Y:=X/X(H-1)= \bigoplus_{ i < j}e_{i,j} \mathbb Z [G/H]$ and denote by $\pi : X \rightarrow Y$ the natural quotient map such that $\pi(e_{i,j} \cdot aG')= e_{i,j}\cdot aH$ for each $e_{i,j}$ and each $a \in G$. Again there is a unique $\mathbb Z [G/H]$-module homomorphism $h: Y \rightarrow A/W=A/[A,H]$ such that $h(e_{i,j})= [g_i,g_j]W$. We have $f=h \circ \pi$.

Using the commutator identities $[ab,c]=[a,c]^b [b,c]$ and $[c,ab]=[c,b][c,a]^b$ we can write $[g^n_i,g_j^n]= \prod_{s=1}^{n^2} [g_{i},g_{j}]^{z_{i,j,s}}$ for some elements $z_{i,j,s} \in G$.
Define $u_{i,j}= e_{i,j} \cdot \sum_{s=1}^{n^2}z_{i,j,s} G'$, we have $f(u_{i,j})=[g_i^n,g_j^n]$.

We have \[ L'W= \langle  [g_i^n,g_j^n] \ | \ 1 \leq  i <j  \leq d \rangle W=f(X(H-1))+ f(\sum_{1 \leq i<j \leq n} \mathbb Z u_{i,j})\] and from $f=h \circ \pi$ it follows that
\[ \frac{A}{L'W} \simeq  \frac{Y} { \pi( \ker f) + \sum_{i<j} \mathbb Z \pi(u_{i,j})}. \]

We consider the $l^1$ norm $|-|_X$ on the free $\mathbb Z$-module $X$ (respectively  the $l^1$ norm $|-|_Y$ on $Y$) as the sum of the absolute values of coordinates computed with respect to the standard $\mathbb Z$-basis $\{e_{i,j} \bar g \  | \ i< j, \bar g \in G^{ab} \}$ of $X$ (respectively the $\mathbb Z$-basis $\{e_{i,j} \tilde g \  | \ i<j, \ \tilde g \in G/H\}$ of $Y$). Note that for any element $v \in X$ we have $|v|_X \geq |\pi(v)|_Y$ and also $|v|_X= |v \cdot \bar g |_X$ for any $\bar g \in G^{ab}$.

Let $c=\max \{ |r_1|_X, \ldots, |r_k|_X\}$ where $r_1, \ldots r_k$ are the $\mathbb Z[G^{ab}]$ module generators of $\ker f$. Then $\pi(\ker f)$ is generated as $\mathbb Z$-module by the set 
$T:=\{ \pi (r_s) \tilde g  \  | \ s=1,\ldots, k, \ \tilde g \in G/H \}$.
Observe that $|\pi(u_{i,j})|_Y \leq |u_{i,j}|_X \leq n^2$ for each $1 \leq i<j \leq d$ while for each $w=\pi (r_s) \tilde g$ we have $|w|_Y \leq |r_s|_X \leq c$. \medskip

The group $(Y,+)$ is a free $\mathbb Z$-module of rank $nd(d-1)/2$ and  $A/L'W \simeq  Y/Z$ where \[ Z =\pi( \ker f) + \sum_{1 \leq i<j \leq d} \mathbb Z \pi(u_{i,j})= \sum_{w \in T} \mathbb Z w + \sum_{1 \leq i <j \leq d}\mathbb Z \pi(u_{i,j}).\]

 Lemma \ref{torsion} applied to $Y/Z$ gives 
$t(A/L'W) \leq (n^2)^{d(d-1)/2}c^{nd(d-1)/2}$ and so 
\[ t(H/H') \leq t(G^{ab}) t(A/L'W) \leq  t(G^{ab}) (n^2)^{d(d-1)/2}c^{nd(d-1)/2} \]
Since $t(G^{ab})$ does not depend on $H$ and $n^2 < 3^n$ we may take $C=(t(G^{ab})3c)^{d^2/2}$ giving 
$t(H/H') < C^n$ as required.
\end{proof}

\end{document}